\documentclass[12pt,leqno]{article}
\usepackage{geometry,amsmath,amsthm,amssymb}
\usepackage[blocks]{authblk}
\usepackage{thmtools}
\usepackage[colorlinks=true,linkcolor=blue,citecolor=blue,pdfpagelabels=false]{hyperref}

\theoremstyle{plain}
  \declaretheorem[numberwithin=section]{theorem}
  \declaretheorem[numberlike=theorem]{corollary}
  \declaretheorem[numberlike=theorem]{proposition}

\theoremstyle{definition}
  \declaretheorem[numberlike=theorem]{example}

\newcommand{\sgn}{\operatorname{sgn}}
\newcommand{\mathd}{\mathrm{d}}
\newcommand{\assign}{:=}
\newcommand{\um}{-}

\makeatletter
\newcommand{\subjclass}[2][1991]{%
  \let\@oldtitle\@title%
  \gdef\@title{\@oldtitle\footnotetext{#1 \emph{Mathematics subject classification.} #2}}%
}
\newcommand{\keywords}[1]{%
  \let\@@oldtitle\@title%
  \gdef\@title{\@@oldtitle\footnotetext{\emph{Key words and phrases.} #1.}}%
}
\makeatother

\begin{document}

\title{Bounds for the logarithm of the Euler gamma function and its derivatives}

\author{Harold G.~Diamond\thanks{hdiamond@illinois.edu}{}\; and Armin Straub\thanks{astraub@illinois.edu}}
\affil{Department of Mathematics\\ University of Illinois at Urbana-Champaign\\ \small 1409 W. Green St, Urbana, IL 61801, United States}

\date{August 13, 2015}


\maketitle

\begin{abstract}
  We consider differences between $\log \Gamma(x)$ and truncations of
  certain classical asymptotic expansions in inverse powers of
  $x-\lambda$ whose coefficients are expressed in terms of Bernoulli
  polynomials $B_n(\lambda)$, and we obtain conditions under which
  these differences are strictly completely monotonic.  In the
  symmetric cases $\lambda=0$ and $\lambda=1/2$, we recover results of
  Sonin, N\"orlund and Alzer.  Also we show how to derive these
  asymptotic expansions using the functional equation of the
  logarithmic derivative of the Euler gamma function, the
  representation of $1/x$ as a difference $F(x+1)-F(x)$, and a
  backward induction.
\end{abstract}

\section{Introduction} \label{sec:intro}

The Euler gamma function
$\Gamma(x)$ is widely regarded as the most important special function.
Good accounts and many formulas for $\Gamma(x)$ can be found in
\cite{Ar}, \cite[\S6.3]{AS}, \cite[Chapter~1]{EMOT}, and \cite[Chapter~12]{WW}.
Various bounds are known for $\Gamma(x)$ and $\psi(x)$, the logarithmic
derivative of $\Gamma(x)$,  also called the \emph{digamma function}, e.g.
\cite{Al1}, \cite{Al2}, \cite[6.3.21]{AS}, \cite{Go},
\cite[Appendix~C]{MV}, \cite{Mo}.
The logarithm of the gamma function has the classical asymptotic expansion
\begin{align}\label{eq:logG:asy}
  \log \Gamma ( x) & \sim \log \sqrt{2 \pi} + \left( x - \tfrac{1}{2}
  \right) \log ( x - \lambda) - ( x - \lambda)\\
  &\qquad + \sum_{n = 2}^{\infty} ( - 1)^n \frac{B_n ( \lambda)}{n ( n - 1)} ( x
  - \lambda)^{1 - n},\nonumber
\end{align}
which we consider for real $x>\lambda$ and $\lambda\in[0,1]$.
As usual, $B_n(\lambda)$ denotes the $n$th Bernoulli function.
Our focus here is on differences (and their derivatives) between
$\log\Gamma(x)$ and the truncations \begin{align*}
  L_N ( \lambda ; x) & := \log \sqrt{2 \pi} + \left( x - \tfrac{1}{2} \right)
  \log ( x - \lambda) - ( x - \lambda)\\
  &\qquad + \sum_{n = 2}^N ( - 1)^n \frac{B_n ( \lambda)}{n ( n - 1)} ( x -
  \lambda)^{1 - n}
\end{align*}
of this asymptotic series. In the most basic case, we are interested in
conditions under which these truncations provide strict upper or lower bounds.
These inequalities together with the functional equations provide an
effective method for calculating $\log\Gamma(x)$ and its derivatives with arbitrary accuracy.

\begin{example}\label{eg:psi:intro}
  A pleasant property of \eqref{eq:logG:asy} is that we can differentiate both
  sides to obtain asymptotic expansions for the derivatives of $\log\Gamma(x)$,
  that is, the derivatives of the truncations $L_N(\lambda;x)$ provide asymptotic
  expansions for the derivatives of $\log\Gamma(x)$.
  For instance, in the particularly symmetric case $\lambda=1/2$, we have the
  asymptotic expansion
  \begin{equation}\label{eq:psi:asy2}
    \psi(x) \sim \log(x-1/2) - \sum_{n=1}^\infty
    \frac{B_{2n}(1/2)}{2n}\,(x-1/2)^{-2n}.
  \end{equation}
  Our main theorem implies the following result of N\"orlund \cite[\S 56,
  p.~106]{No}: if the series in \eqref{eq:psi:asy2} is truncated at $n=N$
  with $N \ge 0$ and even, the resulting sum is a lower bound for $\psi(x)$
  valid for all $x > 1/2$; if $N$ is odd, it is an upper bound.
  In other words, successive truncations of \eqref{eq:psi:asy2} provide
  lower and upper bounds valid for all $x > 1/2$.
  The first examples of these bounds are
  \begin{equation}   \label{eq:ex-1}
  \psi(x) > \log (x-1/2)   
  \end{equation}
  and
  \begin{equation}   \label{eq:ex-2} 
  \psi(x) < \log (x-1/2) + (x-1/2)^{-2}/24\,.
  \end{equation}
\end{example}

A function $f$ is said to be {\emph{completely monotonic}} on an
interval or half-line $I$ if, for all $x \in I$ and all integers $n
\geq 0$, the derivative $f^{( n)} ( x)$ exists and satisfies $( - 1)^n
f^{( n)} ( x) \geq 0$. If this inequality is strict, that is, if $( -
1)^n f^{( n)} ( x) > 0$, then $f$ is called {\emph{strictly
completely monotonic}} on $I$. For an introduction to completely
monotonic functions and their properties we refer to \cite[Chapter~IV]{Wi}.

Our main result, which is proved in Section~\ref{sec:scm}, is the following.
As indicated in Example~\ref{eg:alzer}, this is an extension of a result of
Alzer \cite[Theorem~8]{Al1}, which concerns the case $\lambda = 0$.

\begin{theorem}\label{thm:main}
  Let $N \geq 1$. Suppose that either $N$ is even and $\lambda \in [ 0,
  \lambda_0]$ where $\lambda_0$ is the unique root of $B_N ( \lambda)$ in $[
  0, 1 / 2]$, or $N$ is odd and $\lambda \in [ \lambda_0, 1 / 2]$ where
  $\lambda_0$ is the unique root of $B_{N + 1} ( \lambda)$ in $[ 0, 1 / 2]$.
  Then,
  \begin{equation*}
    ( - 1)^{\lceil N / 2 \rceil} ( \log \Gamma ( x) - L_N ( \lambda ; x))
  \end{equation*}
  is strictly completely monotonic on $( \lambda, \infty)$.
\end{theorem}

We note that $\lambda_0<1/4$ and that $\lambda_0\to1/4$ very rapidly
as $N$ increases.  This is explained in Example~\ref{eg:psi:14}, which also
illustrates that Theorem~\ref{thm:main} would be false without the conditions
on $N$ and $\lambda$.  On the other hand, the restriction to the case
$\lambda\in[0,1/2]$ is made for the sake of exposition, and we invite the
interested reader to adjust the details for the case $\lambda\in[1/2,1]$.

\begin{example}
  The particular case $\lambda = 1 / 2$ and $N = 3$ shows that
  \begin{equation*}
    \log \Gamma ( x) - \log \sqrt{2 \pi} - \left( x - \tfrac{1}{2} \right)
     \left( \log \left( x - \tfrac{1}{2} \right) - 1 \right) + \frac{1 / 24}{x
     - \tfrac{1}{2}}
  \end{equation*}
  is strictly completely monotonic on $( 1 / 2, \infty)$. This generalizes the
  inequality \eqref{eq:ex-2}. 
\end{example}

\begin{example}
  \label{eg:alzer}In the case $\lambda = 0$, we find that the functions $\psi
  ( x) - L_{2 M} ( 0 ; x)$, explicitly given by
  \begin{equation*}
    \psi ( x) - \log \sqrt{2 \pi} - \left( x - \tfrac{1}{2} \right) \log ( x)
     + x - \sum_{n = 1}^M \frac{B_{2 n}}{2 n ( 2 n - 1) x^{2 n - 1}},
  \end{equation*}
  are strictly completely monotonic on $( 0, \infty)$ if $M$ is even.
  Likewise, the functions $L_{2 M} ( 0 ; x) - \psi ( x)$ are strictly
  completely monotonic on $( 0, \infty)$ if $M$ is odd. This is a result of
  Alzer \cite[Theorem~8]{Al1}.
\end{example}

This paper is organized as follows. In Section~\ref{sec:psi}, we give a first
illustration of our approach by proving the claims made in
Example~\ref{eg:psi:intro}.  These bounds for the digamma function are then
generalized in Section~\ref{sec:psi:x}, before we prove the main result in
Section~\ref{sec:scm}.  A very brief indication of how to use these bounds for
numerical calculations is then given in Section~\ref{sec:calc}.

\section{Proving bounds for the digamma function}\label{sec:psi}

In this section, we illustrate our approach by proving the claims made in
Example~\ref{eg:psi:intro}.
In particular, we show, by a method we believe new, how to derive asymptotic
expansions for $\psi(x)$, such as \eqref{eq:psi:asy2}, whose successive
truncations provide lower and upper bounds.

Assuming there is an approximation of the form $\psi(x) \approx F(x)$,
then the functional equation $\psi(x+1) = \psi(x) + 1/x$ would give
\begin{equation} \label{eq:F-relation}
F(x+1) - F(x) \approx 1/x.
\end{equation}
Our method is to find functions $F$ giving upper (resp.~lower)
bounds in \eqref{eq:F-relation} and then show by a ``backward
induction'' that these functions provide bounds for $\psi(x)$.  We use
one additional fact about the digamma function $\psi(x)$, namely the
asymptotic formula 
\begin{equation}  \label{eq:psi-log}
\psi(x) = \log x + o(1), \quad x \to \infty.
\end{equation}
For completeness, we prove this at the end of the article.

\begin{example}\label{eg:lower}
  To illustrate our method, we first obtain the lower bound
  \eqref{eq:ex-1}.  By convexity,
  \[
    \log (x+1/2) - \log (x-1/2) = \int_{x-1/2}^{x+1/2} \frac{\mathd u}{u} > 
    \frac 1x\,. 
  \]
  This inequality and the functional equation give
  \[
    \psi(x) > \psi(x+1) - (\log (x+1/2) - \log(x-1/2) )
  \]
  for all $x>1/2$, or, equivalently,
  \begin{equation}  \label{eq:functional-ineq}
    \psi(x) - \log(x-1/2) > \psi(x+1) - \log (x+1/2).
  \end{equation}

  Now we prove \eqref{eq:ex-1} by an induction: repeating
  \eqref{eq:functional-ineq} with $x$ replaced by $x+1, \, x+2, \,
  \dots$ and using \eqref{eq:psi-log}, we see that
  \[
    \psi(x) - \log(x-1/2) > \liminf_{y \to \infty}(\psi(y)-\log(y-1/2)) = 0.
  \]
\end{example}

The key to our treatment of the general case is to develop a suitable
expression for $1/x$ with the aid of the Euler--Maclaurin sum formula.
The following version of Euler--Maclaurin, given in \cite[Appendix~B]{MV} is
convenient for our application.

\begin{theorem}[Euler--Maclaurin]\label{th:EM}
  Let $K$ be a positive integer,
  and suppose that $f$ is $K$ times continuously differentiable on $[ a, b]$.
  Let $B_k(\cdot)$ denote the $k$-th Bernoulli function.
  Then,
  \begin{align*}
    \sum_{a < n \leq b} f ( n) & = \int_a^b f ( x) \mathd x\\
    &\quad + \sum_{k = 1}^K \frac{( - 1)^k}{k!} ( B_k ( \{ b \}) f^{( k - 1)} (
    b) - B_k ( \{ a \}) f^{( k - 1)} ( a))\\
    &\quad - \frac{( - 1)^K}{K!} \int_a^b B_K ( \{ x \}) f^{( K)} ( x) \mathd x,
  \end{align*}
  where $\{t\} := t - \lfloor t \rfloor$.
\end{theorem}

We now show the following result of N\"orlund \cite[\S 56, p.~106]{No} stated in Example~\ref{eg:psi:intro}.

\begin{theorem}\label{thm:norlund}
  The asymptotic expansion \eqref{eq:psi:asy2} for $\psi(x)$ holds as $x\to\infty$.
  Moreover, if the series in \eqref{eq:psi:asy2} is truncated at $n=N$ with $N
  \ge 0$ and even, the resulting sum is a lower bound for $\psi(x)$ valid for
  all $x > 1/2$; if $N$ is odd, it is an upper bound.
\end{theorem}

\begin{proof}
  Take $f(t) := 1/(x+t)$, $a = -1/2$, $b = 1/2$ in Theorem~\ref{th:EM}. 
  Because $B_k(1/2) = 0$ for all odd $k$, we restrict the series to
  even indices.  Starting with the one-term (!) sum
\[
\frac 1x = \!\! \sum_{-1/2 < n \le 1/2} \frac 1{x+n}\,,
\]
we have
\[
\frac 1x= \int_{-1/2}^{1/2} \frac{\mathd t}{x+t} 
-\sum_{n=1}^N \frac{B_{2n}(1/2)}{2n} \,\Big(\frac 1{(x+1/2)^{2n}} -
\frac 1{(x-1/2)^{2n}}\Big) + R_N,
\]
where
\[
R_N := -\int_{-1/2}^{1/2}B_{2N}(\{t\})\frac{\mathd t}{(x+t)^{2N+1}}\,.
\]

For the claimed inequalities, we establish the sign of $R_N$ 
with the aid of several properties of Bernoulli functions, all of
which are given in \cite{MV}.  For $-1/2 < t <0$,
\[
B_{2N}(\{t\}) = B_{2N}(1+t) =  B_{2N}(-t).
\]
Thus, we have
\[
R_N = -\int_0^{1/2}\!\!B_{2N}(t) \Big(\frac{1}{(x-t)^{2N+1}}
+\frac{1}{(x+t)^{2N+1}}\Big)\,\mathd t\,.
\]
Next,
\[
B_{2N+1}(y) = (2N+1)\!\int_0^y B_{2N}(t)\, \mathd t.
\]
Integrating by parts and noting that the integrated term vanishes
because $B_{2N+1}(1/2) = B_{2N+1}(0) = 0$, we find
\[
R_N = \int_0^{1/2}\!\!B_{2N+1}(t) \Big(\frac{1}{(x-t)^{2N+2}}
-\frac{1}{(x+t)^{2N+2}}\Big)\,\mathd t\,.
\]
For $0 < t< 1/2, \ \sgn(B_{2N+1}(t)) = (-1)^{N+1}$.  Also, in this range,
\[
(x-t)^{-2N-2} - (x+t)^{-2N-2} >0
\] 
by monotonicity.  Thus $R_N < 0$ for $N$ even and $R_N > 0$ for $N$
odd.

Setting
\begin{equation} \label{eq:F_N}
F_N(x) := \log(x-1/2) -  \sum_{n=1}^N \frac{B_{2n}(1/2)}{2n}\,(x-1/2)^{-2n},
\end{equation}
we find
\[
\psi(x+1) -  \psi(x) = \frac 1x  =  F_N(x+1) - F_N(x) + R_N.
\]
Thus
\[
\psi(x)  - F_N(x) > \psi(x+1) - F_N(x+1)
\]
for $N$ even (resp. $<0$ for $N$ odd).

Assuming that $N$ is even, we proceed as in Example~\ref{eg:lower}.  We
find inductively that, for all $x > 1/2$,
\[
\psi(x)  - F_N(x) > \liminf_{y \to \infty} (\psi(y)  - F_N(y)) = 0,
\]
which proves the lower bound for $\psi(x)$.  An analogous argument shows
that $\psi(x) - F_N(x) < 0$ for $N$ odd. 
\end{proof}

\begin{example}\label{eg:stirling}
  Integrating formula \eqref{eq:psi:asy2} 
  termwise, we find, for a suitable constant $C$, 
  \begin{align*}
  \log \Gamma(x) \sim C + &(x-1/2)\log(x-1/2) - (x - 1/2) \\
  &+ \sum_{n=1}^\infty\frac{B_{2n}(1/2)}{2n(2n-1)}\,(x-1/2)^{1-2n}.
  \end{align*}
  We next exponentiate and then determine $C$ by the well-known device
  of using the duplication formula for the gamma function to connect
  $\Gamma(2x)$ with $\Gamma(x+1/2)$ and $\Gamma(x)$.  Letting $x \to
  \infty$, the terms in the sum go to zero, and we find $C = \log
  \sqrt{2\pi}$.

  As a consequence of Theorem~\ref{thm:norlund}, we obtain
  that, for any positive integer $N$ and all $x>1/2$,
  \begin{align}\label{eq:gamma:bounds}
    \Gamma_{2N-1}(x) < \Gamma(x) < \Gamma_{2N}(x),
  \end{align}
  where
  \begin{align*}
  \Gamma_N(x) &= \sqrt{2\pi}\, (x -1/2)^{x-1/2}\, e^{-(x - 1/2)} \\
  &\qquad \times
  \exp\bigg\{\sum_{n=1}^{N}\frac{B_{2n}(1/2)}{2n(2n-1)}
  \,(x-1/2)^{1-2n} \bigg\}.
  \end{align*}
  These bounds for the gamma function are given by N\"orlund \cite[\S 58,
  p.~111]{No}, who attributes them to Sonin.

  Note that the inequalities in \eqref{eq:gamma:bounds} are opposite from those of
  Theorem~\ref{thm:norlund}, because the integration has changed the signs in
  the summations. 
\end{example}

\section{Generalized bounds for the digamma function}\label{sec:psi:x}

Let $\lambda \in [0,\,1]$.
As $x \rightarrow \infty$, we have also the asymptotic expansion
\begin{equation}
  \psi (x) \sim \log (x - \lambda) - \sum_{n = 1}^{\infty} (- 1)^n \frac{B_n
  (\lambda)}{n} ( x - \lambda)^{- n} . \label{eq:psi:a:asy}
\end{equation}
This generalization of the formula \eqref{eq:psi:asy2} in Theorem~\ref{thm:norlund}  
can be derived by our method. It appears also, e.g., in \cite[\S2.11, p.~33]{Lu}.
Note that here, except at $\lambda = 0, 1 / 2, 1$,
the Bernoulli functions with odd indices are nonzero. For $\lambda =
0$, $B_n ( 0) = B_n$, the $n$th Bernoulli number, and \eqref{eq:psi:a:asy} 
becomes the well-known asymptotic expansion \cite[6.3.18]{AS}
\begin{equation}
  \psi (x) \sim \log x - \frac{1}{2 x} - \sum_{n = 1}^{\infty} \frac{B_{2
  n}}{2 n} x^{- 2 n} . \label{eq:classical-psi}
\end{equation}
At first glance, this formula looks very similar to the expansion \eqref{eq:psi:asy2}.
However, in addition to having an extra term and a
different center of expansion, the coefficients of the last series are
Bernoulli numbers rather than Bernoulli functions evaluated at 1/2; the
connection between these is given by the formula $B_n  (1 / 2) = - (1 - 2^{1 -
n}) B_n$ for $n = 0, 1, \ldots$ \cite[23.1.21]{AS}.

Denote the truncation of the asymptotic series in
\eqref{eq:psi:a:asy} at $n = N$ by
\begin{equation}\label{eq:FN}
  F_N ( \lambda ; x) := \log (x - \lambda) - \sum_{n = 1}^N (- 1)^n \frac{B_n
   (\lambda)}{n} ( x - \lambda)^{- n} .
\end{equation}
The following generalization of Theorem~\ref{thm:norlund} specifies
conditions under which these approximations to $\psi ( x)$ provide
lower or upper bounds that are valid for all $x > \lambda$.
As in the case of our main result, Theorem~\ref{thm:main}, which is a natural
extension of these bounds, the restriction to the case $\lambda\in[0,1/2]$ is
made for expository reasons.

\begin{theorem}
  \label{thm:psi:bounds:x}Let $N\ge1$. Let $\lambda_0$ be the unique root of $B_N (
  \lambda)$ in $[ 0, 1 / 2]$ if $N$ is even, and the unique root of $B_{N + 1}
  ( \lambda)$ in $[ 0, 1 / 2]$ if $N$ is odd.
  \begin{itemize}
    \item Suppose that $N \equiv 1$ modulo $4$ and $\lambda \in [ \lambda_0, 1
    / 2]$. Then, we have $\psi ( x) > F_N ( \lambda ; x)$ for all $x >
    \lambda$. If $N \equiv 3$ modulo $4$, then the inequality for $\psi$ is
    reversed, that is, $\psi ( x) < F_N ( \lambda ; x)$.
    
    \item Suppose that $N \equiv 2$ modulo $4$ and $\lambda \in [ 0,
    \lambda_0]$. Then, we have $\psi ( x) > F_N ( \lambda ; x)$ for all $x >
    \lambda$. If $N \equiv 0$ modulo $4$, then the inequality for $\psi$ is
    reversed.
  \end{itemize}
\end{theorem}

We note that $\lambda_0<1/4$ and that $\lambda_0\to1/4$ very rapidly
as $N$ increases; see Example~\ref{eg:psi:14}. 

In preparation for the proof of Theorem~\ref{thm:psi:bounds:x}, we deduce the
following approximate version of the functional equation $\psi ( x + 1) - \psi
( x) = 1 / x$.

\begin{proposition}
  \label{prop:FN:R}Let $N\ge1$. Then
  \begin{equation*}
    F_N ( \lambda ; x + 1) - F_N ( \lambda ; x) = \frac{1}{x} + \int_{-
     \lambda}^{1 - \lambda} \frac{B_N ( \{ t \})}{(x + t)^{N + 1}} \mathd t.
  \end{equation*}
\end{proposition}

\begin{proof}
  As in the case $\lambda = 1 / 2$, we set $f (t) := 1 / (x + t)$ and apply
  the Euler--Maclaurin summation formula to the one-term sum
  \begin{align*}
    \frac{1}{x} = \sum_{- \lambda < n \leq 1 - \lambda} \frac{1}{x + n} .
  \end{align*}
  We find
  \begin{align*}
    \frac{1}{x} &= \int_{- \lambda}^{1 - \lambda} \frac{1}{x + t} \mathd x -
     \sum_{n = 1}^N \frac{1}{n} \left( \frac{B_n ( 1 - \lambda)}{(x + 1 -
     \lambda)^n} - \frac{B_n ( 1 - \lambda)}{(x - \lambda)^n} \right) \\
     &\quad- \int_{- \lambda}^{1 - \lambda} \frac{B_N ( \{ t \})}{(x + t)^{N + 1}}
     \mathd t
  \end{align*}
  and it remains only to note that $B_n ( 1 - \lambda) = ( - 1)^n B_n (\lambda)$.
\end{proof}

\begin{proof}[Proof of Theorem~\ref{thm:psi:bounds:x}]
  First, assume that $N \equiv 1$
  modulo $4$ and that $\lambda \in [ \lambda_0, 1 / 2]$. In light of
  Proposition~\ref{prop:FN:R} and the backward induction already used in the
  case $\lambda = 1 / 2$, it suffices to show that
  \begin{equation*}
    R_N \assign \int_{- \lambda}^{1 - \lambda} \frac{B_N ( \{ t \})}{(x +
     t)^{N + 1}} \mathd t > 0.
  \end{equation*}
  To determine the sign, we split this integral as
  \begin{equation*}
    R_N = R' + R'' \assign \int_{- \lambda}^{\lambda} \frac{B_N ( \{ t
     \})}{(x + t)^{N + 1}} \mathd t + \int_{\lambda}^{1 - \lambda} \frac{B_N (
     \{ t \})}{(x + t)^{N + 1}} \mathd t.
  \end{equation*}
  For $t \in ( - \lambda, 0)$, we have $B_N ( \{ t \}) = B_N ( 1 + t) = - B_N
  ( - t)$ since $N$ is odd, and hence
  \begin{align*}
    R' & = - \int_{- \lambda}^0 \frac{B_N ( - t)}{(x + t)^{N + 1}} \mathd t
    + \int_0^{\lambda} \frac{B_N ( t)}{(x + t)^{N + 1}} \mathd t\\
    & = - \int_0^{\lambda} B_N ( t) \left( \frac{1}{(x - t)^{N + 1}} -
    \frac{1}{(x + t)^{N + 1}} \right) \mathd t.
  \end{align*}
  Integrating by parts and observing that the boundary terms vanish, we find
  \begin{equation*}
    R' = \int_0^{\lambda} ( B_{N + 1} ( t) - B_{N + 1} ( \lambda)) \left(
     \frac{1}{(x - t)^{N + 2}} + \frac{1}{(x + t)^{N + 2}} \right) \mathd t.
  \end{equation*}
  Likewise, splitting $R''$ according to $\int_{\lambda}^{1 - \lambda} =
  \int_{\lambda}^{1 / 2} + \int_{1 / 2}^{1 - \lambda}$, then integrating by
  parts, we obtain the counterpart
  \begin{equation*}
    R'' = \int_{\lambda}^{1 / 2} ( B_{N + 1}  ( t) - B_{N + 1} ( \lambda))
     \left( \frac{1}{( x + t)^{N + 2}} + \frac{1}{( x + 1 - t)^{N + 2}}
     \right) \mathd t.
  \end{equation*}
  In order to bound $R_N = R' + R''$, we determine, in both integrals, the
  signs of each of the two factors. Since $N \equiv 1$ modulo $4$, the
  polynomial $B_{N + 1} ( t)$ is decreasing for $0 \leq t \leq 1 /
  2$ \cite[(B.12)]{MV}. It follows that $B_{N + 1}  ( t) - B_{N + 1} (
  \lambda) \geq 0$ for $0 \leq t \leq \lambda$, and $B_{N + 1} 
  ( t) - B_{N + 1} ( \lambda) \leq 0$ for $\lambda \leq t \leq
  1 / 2$. Since the function
  \begin{equation*}
    \frac{1}{( x - t)^{N + 2}} + \frac{1}{( x + t)^{N + 2}}
  \end{equation*}
  is increasing and positive for $t$ such that $0 \leq t < x$, we find
  that
  \begin{equation*}
    R' > \frac{2}{x^{N + 2}} \int_0^{\lambda} ( B_{N + 1}  ( t) - B_{N + 1}
     ( \lambda)) \mathd t.
  \end{equation*}
  On the other hand, for all $t \in [0,\,1]$,
  \begin{equation*}
    \frac{1}{( x + t)^{N + 2}} + \frac{1}{( x + 1 - t)^{N + 2}} <
     \frac{2}{x^{N + 2}}\,,
  \end{equation*}
  which allows us to bound
  \begin{equation*}
    R'' > \frac{2}{x^{N + 2}} \int_{\lambda}^{1 / 2} ( B_{N + 1}  ( t) - B_{N
     + 1} ( \lambda)) \mathd t.
  \end{equation*}
  Combining the two inequalities, we obtain
  \begin{equation*}
    R_N > \frac{2}{x^{N + 2}} \int_0^{1 / 2} ( B_{N + 1}  ( t) - B_{N + 1} (
     \lambda)) \mathd t = - \frac{B_{N + 1} ( \lambda)}{x^{N + 2}} .
  \end{equation*}
  Since $\lambda_0$ is the zero of $B_{N + 1} ( \lambda)$, which is decreasing
  on $[ 0, 1 / 2]$, and since $\lambda \geq \lambda_0$, we conclude that
  $R_N > 0$. It is straightforward to adjust this proof to cover also the case
  $N \equiv 3$ modulo $4$.
  
  Next, assume that $N \equiv 2$ modulo $4$ and that $\lambda \in [ 0,
  \lambda_0]$. Again, it suffices to show that
  \begin{equation*}
    R_N \assign \int_{- \lambda}^{1 - \lambda} \frac{B_N ( \{ t \})}{(x +
     t)^{N + 1}} \mathd t > 0.
  \end{equation*}
  Let $\ell ( t)$ be the linear function whose graph intersects the graph of
  $1/(x + t)^{N + 1}$ at $t = \lambda_0$ and $t = 1 - \lambda_0$. Since $B_N (
  \{ t \}) < 0$ for $\lambda_0 < t < 1 - \lambda_0$, and $B_N ( \{ t \}) > 0$
  for $\um \lambda_0 < t < \lambda_0$ or $1 - \lambda_0 < t < 1 + \lambda_0$,
  it follows that
  \begin{equation*}
    B_N ( \{ t \}) \left[ \frac{1}{(x + t)^{N + 1}} - \ell ( t) \right]
     \geq 0
  \end{equation*}
  for all $t \in [ - \lambda_0, 1 + \lambda_0]$, with strict inequality if $t
  \neq \lambda_0$ and $t \neq 1 - \lambda_0$. In particular, writing $\ell (
  t) = m t + b$, with $m < 0$, we have
  \begin{equation}
    R_N > \int_{- \lambda}^{1 - \lambda} B_N ( \{ t \}) ( m t + b) \mathd t =
    m \int_{- \lambda}^{1 - \lambda} t B_N ( \{ t \}) \mathd t.
    \label{eq:RN:2}
  \end{equation}
  The last equality is a consequence of the fact that the average value 
  of $B_N(\{t\})$ over a period is zero. Splitting the last integral as 
  $\int_{- \lambda}^{1 - \lambda} = \int_{- \lambda}^0 +
  \int_0^{1 - \lambda}$, we obtain
  \begin{equation*}
    \int_{- \lambda}^{1 - \lambda} t B_N ( \{ t \}) \mathd t = \int_{-
     \lambda}^0 t B_N ( 1 + t) \mathd t + \int_0^{1 - \lambda} t B_N ( t)
     \mathd t.
  \end{equation*}
  After a change of variables in the first integral and the observation
  that $\int_0^1 t B_N ( t) \mathd t=0$, this simplifies to
  \begin{equation*}
    \int_{- \lambda}^{1 - \lambda} t B_N ( \{ t \}) \mathd t = \int_0^1 t B_N
     ( t) \mathd t - \int_{1 - \lambda}^1 B_N ( t) \mathd t = - \int_{1 -
     \lambda}^1 B_N ( t) \mathd t < 0,
  \end{equation*}
  Since $m < 0$, we conclude from \eqref{eq:RN:2} that $R_N > 0$.
  Again, it is straightforward to adjust the proof to cover the case
  when $N \equiv 0$ modulo $4$.
\end{proof}

\begin{example}\label{eg:psi:14}
  The equality of the generating functions
  \begin{align*}
    \sum_{n=0}^\infty B_{2n}\left( \tfrac14 \right) \frac{x^{2n}}{(2n)!}
    = \frac12 \left( \frac{x e^{x/4}}{e^x-1} - \frac{x e^{-x/4}}{e^{-x}-1} \right)
    = \frac12 \frac{x e^{x/4}}{e^{x/2}-1}
    = \sum_{n=0}^\infty \frac{B_n\left( \tfrac12 \right)}{2^n} \frac{x^n}{n!}
  \end{align*}
  implies that $B_{2N} ( 1 / 4) = 2^{- 2N} B_{2N} ( 1 / 2)$. Since
  $\sgn B_{2N}(1/4) = \sgn B_{2N}(1/2)$ and  $B_{2N}(\lambda)$ is monotonic
  on $(0,1/2)$, we conclude 
  that the unique root $\lambda_0$ of the Bernoulli polynomial $B_{2N}
  ( \lambda)$ in $[ 0, 1 / 2]$ satisfies $\lambda_0 < 1 / 4$.
  Theorem~\ref{thm:psi:bounds:x} therefore implies that, for $x > 1 / 4$,
  \begin{align*}
    \psi ( x) & > F_1 ( 1 / 4 ; x) = \log (x - \tfrac{1}{4}) - \frac{1 /
    4}{x - \tfrac{1}{4}}\,,\\
    \psi ( x) & < F_3 ( 1 / 4 ; x) = \log (x - \tfrac{1}{4}) - \frac{1 /
    4}{x - \tfrac{1}{4}} + \frac{1 / 96}{(x - \tfrac{1}{4})^2} + \frac{1 /
    64}{(x - \tfrac{1}{4})^3}\,,
  \end{align*}
  and so on.
  On the other hand, for instance, the difference $\psi(x) -
  F_2(1/4;x)$ changes sign
  at some $x$ in the interval $(1/4,\infty)$. 
\end{example}

\begin{example}
  Applying Theorem~\ref{thm:psi:bounds:x} in the special case $\lambda = 0$,
  we find that truncating the classical asymptotic series
  \eqref{eq:classical-psi} results in strict upper and lower bounds for $\psi
  ( x)$, valid for all $x > 0$. The first two such bounds, namely
  \begin{align*}
    \psi (x) & > \log (x) - \frac{1}{2 x} - \frac{1}{12 x^2}\,,\\
    \psi (x) & < \log (x) - \frac{1}{2 x} - \frac{1}{12 x^2} + \frac{1}{120 x^4}
  \end{align*}
  are proved also in \cite[Theorem~5]{Go}.  That these inequalities 
  hold more generally is implied by \cite[Theorem~8]{Al1}.
\end{example}

\section{Complete monotonicity}\label{sec:scm}

We are now ready to generalize our approach to give a proof of
Theorem~\ref{thm:main}, our main result.
First, we note that, combining previous observations, it is straightforward to
make the error term in the asymptotic expansion \eqref{eq:logG:asy} explicit.
The special case $\lambda = 0$ of \eqref{eq:loggamma:err} appears, for
instance, in \cite[Theorem~D.3.2]{AAR}, and equation \eqref{eq:psi:err} is
given in \cite[\S 53, p.~101, (6)]{No}.

\begin{proposition}
  \label{prop:loggamma:err}Let $N \geq 1$. Then for $x > \lambda$
  \begin{equation}
    \log \Gamma ( x) - L_N ( \lambda ; x) = - \frac{1}{N} \int_{-
    \lambda}^{\infty} \frac{B_N ( \{ t \})}{( x + t)^N} \mathd t.
    \label{eq:loggamma:err}
  \end{equation}
\end{proposition}

\begin{proof}
  This follows by integration, and the fact that $\log \Gamma ( x) - L_N (
  \lambda ; x)$ converges to 0 
  as $x \rightarrow \infty$, if we can show that
  \begin{equation}
    \psi ( x) - F_N ( \lambda ; x) = \int_{- \lambda}^{\infty} \frac{B_N ( \{
    t \})}{( x + t)^{N + 1}} \mathd t, \label{eq:psi:err}
  \end{equation}
  where $F_N ( \lambda ; x)$ is defined as in \eqref{eq:FN}. Note that
  Proposition~\ref{prop:FN:R} can be restated as
  \begin{equation*}
    \psi ( x) - F_N ( \lambda ; x) = \psi ( x + 1) - F_N ( \lambda ; x + 1) +
     \int_{- \lambda}^{1 - \lambda} \frac{B_N ( \{ t \})}{(x + t)^{N + 1}}
     \mathd t.
  \end{equation*}
  Iterating this relation, we find that, for all integers $n \geq 1$,
  \begin{equation*}
    \psi ( x) - F_N ( \lambda ; x) = \psi ( x + n) - F_N ( \lambda ; x + n) +
     \int_{- \lambda}^{n - \lambda} \frac{B_N ( \{ t \})}{(x + t)^{N + 1}}
     \mathd t,
  \end{equation*}
  and equation \eqref{eq:psi:err} follows from taking the limit $n \rightarrow
  \infty$.
\end{proof}

\begin{proof}[Proof of Theorem~\ref{thm:main}]
  Consider, for illustration, the case that $N \equiv 1$ modulo $4$. Under the
  given assumptions, Theorem~\ref{thm:psi:bounds:x} states that $f ( x) = L_N
  ( \lambda ; x) - \log \Gamma ( x)$ satisfies $f' ( x) < 0$. Integrating as
  in Example~\ref{eg:stirling} further shows that $f ( x) > 0$.
  
  We now indicate why, more generally, $( - 1)^n f^{( n)} ( x) > 0$. Note that
  \eqref{eq:loggamma:err} implies that
  \begin{equation*}
    f ( x) - f ( x + 1) = \frac{1}{N} \int_{- \lambda}^{1 - \lambda}
     \frac{B_N ( \{ t \})}{( x + t)^N} \mathd t.
  \end{equation*}
  Differentiating this relation, we observe that
  \begin{equation*}
    f^{( n)} ( x) - f^{( n)} ( x + 1) = ( - 1)^n \frac{( N)_n}{N} \int_{-
     \lambda}^{1 - \lambda} \frac{B_N ( \{ t \})}{(x + t)^{N + n}} \mathd t,
  \end{equation*}
  where $( N)_n = N ( N + 1) \cdots ( N + n - 1) > 0$. By the usual backward
  induction, it follows that $( - 1)^n f^{( n)} ( x) > 0$, if we can show that
  \begin{equation*}
    \int_{- \lambda}^{1 - \lambda} \frac{B_N ( \{ t \})}{(x + t)^{N + n}}
     \mathd t > 0.
  \end{equation*}
  The proof of Theorem~\ref{thm:psi:bounds:x}, which shows this inequality in
  the case $n = 1$, extends in a straightforward manner to show that these
  inequalities hold in general.
\end{proof}

\section{Numerical calculations}\label{sec:calc} 

It is generally infeasible to make a calculation of prescribed accuracy using
an asymptotic expansion.  We can do this for $\psi(x)$ (as well as, more
generally, for $\log\Gamma(x)$ and all of its derivatives), however, because
of the following two special properties.

First, we have explicit error terms for the series \eqref{eq:psi:asy2}.
The corresponding extensions to $\log\Gamma(x)$ and its derivatives (compare
Theorem~\ref{thm:main}) as well as to values of $\lambda$ other than
$\lambda=1/2$ are straightforward.

\begin{corollary} \label{cor:bounds}
Let $F_N(x)$ be defined by \eqref{eq:F_N} and suppose $x > 1/2$. Then
$|\psi(x) - F_N(x)|$ is smaller than the first omitted term of the 
series, and, using the average of consecutive terms,
\[
\left|\psi(x) - \frac{F_N(x) + F_{N+1}(x)}{2}\right| 
\le \frac 12 \,\frac{|B_{2N+2}(1/2)|}{2N+2}\,(x-1/2)^{-2N-2}.
\]\end{corollary}

\begin{proof}
$\psi(x)$ lies between $F_N(x)$ and $F_{N+1}(x)$.  Thus its distance  
from $F_N(x)$ is smaller than $|F_{N+1}(x)-F_N(x)|$, the absolute
value of the last term of $F_{N+1}(x)$.  Further, the distance of
$\psi(x)$ from the average of the two partial sums is at most half 
their separation.
\end{proof}

Second, each fixed partial sum of an asymptotic series of the
type \eqref{eq:psi:asy2} is more
accurate at larger arguments, and by applying the functional equation
repeatedly, we can approximate $\psi$ at a larger argument.

As an illustration, we find $\psi(1)$ (which equals $-\gamma$) to 20
decimal places of accuracy.  For a number $K$ of our choice, write
\[
\psi(1) = \psi(K) - 1 - 1/2 - 1/3 - \dots - 1/(K-1),
\]
and then approximate $\psi(K)$.  Say we take $K = 100$ and use the
first four terms of the series in \eqref{eq:psi:asy2}.  We find
\[
\log 99.5 + \frac{1/24}{99.5^2} - \frac{7/960}{99.5^4}
+ \frac{31/8064}{99.5^6} - \frac{127/30720}{99.5^8}
\]
as an approximation to $\psi(100)$.  The error is guaranteed to be
smaller than the first omitted term of the series, 
\[
(511/67584) \, 99.5^{-10} < 10^{-22}.
\]
We find $\psi(1) \approx -0.57721566490153286061.$ (The analogous
calculation using the first four terms of the series in
\eqref{eq:classical-psi} yields the same result.)
On the other hand, the most accurate approximation to $\psi(1)$, using the
fixed choice $K=100$, is achieved when we use, roughly, the first $300$
terms of the asymptotic series.  In that case, the error is less than
$10^{-272}$.

\section{The asymptotic formula \eqref{eq:psi-log}}
If we logarithmically differentiate the Weierstrass product formula
for $\Gamma(x)$ \cite[6.1.3]{AS} (unfortunately listed there under
\emph{Euler's} name), we get
\[
\psi(x+1) = -\gamma + \sum_{n=1}^\infty \Big( \frac 1n - \frac 1{n+x} \Big).
\]
From this we see that if $M$ is a positive integer and $M-x = O(1)$,
then $\psi(x) = \psi(M) + O(1/M)$.  To prove \eqref{eq:psi-log}, it
suffices to show
\begin{equation}  \label{eq:psi-log-M}
\psi(M+1) = \log M + O(1/M).
\end{equation}

Let $L$ be a large positive integer and write
\[
S_{LM} :=\sum_{n=1}^{LM} \Big( \frac 1n - \frac 1{n+M} \Big)
\]
and group terms in blocks of length $M$.  We find upon rearranging that
\[
S_{LM} =\sum_{n=1}^M\,  \frac 1n - \sum_{n=LM +1}^{(L+1)M} \frac 1{n} 
= \log M + \gamma + O(1/M)  + O(1/L).
\]
Letting $L \to \infty$, \eqref{eq:psi-log-M} and  hence
\eqref{eq:psi-log} is proved.

\end{document}